\documentclass[reqno]{amsart}
\usepackage[english]{babel}
\usepackage{amssymb}
\usepackage[T1]{fontenc}

\newtheorem{theorem}{Theorem}[section]
\newtheorem{proposition}[theorem]{Proposition}
\newtheorem{lemma}[theorem]{Lemma}
\newtheorem{corollary}[theorem]{Corollary}

\theoremstyle{definition}

\theoremstyle{remark}
\newtheorem*{remark}{Remark}
\numberwithin{equation}{section}

\newcommand{\RE}{\mbox{$\mathbb{R}$}}
\newcommand{\CEP}[1]{\mbox{$\mathbb{C}^{#1}$}}

\newcommand{\supp}{\mbox{supp}}
\newcommand{\E}{\mbox{$\mathcal{E}$}}

\newcommand{\Eo}{\mbox{$\mathcal{E}_{0}$}}

\newcommand{\Ep}{\mbox{$\mathcal{E}_{p}$}}

\newcommand{\ddcn}[1]{\mbox{$\left(dd^{c}#1\right)^{n}$}}

\newcommand{\M}{\mbox{$\mathcal{M}$}}

\newcommand{\J}{\mathcal{J}_{\mu}}

\begin{document}

\author{Per \AA hag}
\address{Department of Mathematics and Mathematical Statistics\\ Ume\aa \ University\\ SE-901 87 Ume\aa \\ Sweden}
\email{Per.Ahag@math.umu.se}
\author{Urban Cegrell}
\address{Department of Mathematics and Mathematical Statistics\\ Ume\aa \ University\\ SE-901 87 Ume\aa \\ Sweden}
\email{Urban.Cegrell@math.umu.se}
\author{Rafa\l\ Czy{\.z}}
\address{Institute of Mathematics\\ Jagiellonian University\\ \L ojasiewicza 6\\ 30-348 Krak\'ow\\ Poland}
\email{Rafal.Czyz@im.uj.edu.pl}
\title{On Dirichlet's  principle and problem}

\begin{abstract} The aim of this paper is to give a new proof of the complete characterization of measures for which there exist a solution of the Dirichlet problem for the complex Monge-Amp\`{e}re operator in the set of plurisubharmonic functions with finite pluricomplex energy. The proof uses variational methods.
\end{abstract}

\maketitle

\section{Introduction}
Throughout this note let $\Omega\subseteq\CEP{n}$, $n\geq 1$, be a bounded, connected, open, and hyperconvex set. By
$\Eo$ we denote the family of all bounded plurisubharmonic functions $\varphi$ defined on $\Omega$ such that
\[
\lim_{z\to\xi} \varphi (z)=0\;\; \text{ for every }\;\; \xi\in\partial\Omega\, ,  \;\;\text{ and }\;\; \int_{\Omega} \ddcn{\varphi}<
\infty\, ,
\]
where $\ddcn{\,\cdot\,}$ is the complex Monge-Amp\`{e}re operator. Next let $\Ep$, $p>0$, denote the family of plurisubharmonic functions $u$ defined on $\Omega$ such that there exists a
decreasing sequence $\{u_j\}$, $u_j\in\Eo$, that
converges pointwise to $u$ on $\Omega$, as $j$ tends to $+\infty$, and
\[
\sup_{j\geq 1}\int_{\Omega}(-u_j)^p\ddcn{u_j}=\sup_{j\geq 1} e_p(u_j)< \infty\, .
\]
If $u\in\Ep$, then $e_p(u)<\infty$ (\cite{cegrell_pc,cegrell_sub}). It should be noted that
it follows from~\cite{cegrell_pc} that the complex Monge-Amp\`{e}re operator is well-defined on
$\Ep$. It is not only within pluripotential theory these cones have been proven useful, but also as a tool in dynamical
systems and algebraic geometry (see e.g.~\cite{cegrell_hiep_etc,diller}). For further information on pluripotential theory we refer to~\cite{czyz,klimek,kolo_mem}.

The  purpose of this paper is to give a new proof of Theorem B below and use Theorem B to prove (2)  implies (1) the following theorem:

\bigskip

\noindent {\bf Theorem A (Dirichlet's problem).} \emph{Let $\mu$ be a non-negative Radon measure, then the following conditions are equivalent:}
\begin{enumerate}

\item there exists a function $u\in\E_1$ such that $\ddcn{u}=\mu$,

\item there exists a constant $B>0$, such that
\begin{equation}\label{eq2}
\int_{\Omega}(-\varphi)\, d\mu\leq B\, e_1(\varphi)^{\frac1{n+1}}\text{ for all } \varphi\in\E_1\, ,
\end{equation}

\item the class $\E_1$ is contained in $L^1(\mu)$,

\item the class $\E_1$ is contained in $L^1(\mu)$. Furthermore, for any sequence $\{v_j\}\subset\E_1$ such that
$e_1(v_j)\leq 1$, there exists a subsequence $\{v_{j_k}\}$  convergent in the $L^1(\mu)$ topology.

\end{enumerate}

\bigskip

Originally, it was proved by the second author in ~\cite{cegrell_pc}, that the two first conditions in Theorem A are equivalent. This gives a  complete characterization of measures for which there exist a solution of the Dirichlet problem for the complex Monge-Amp\`{e}re operator in the class $\E_1$.

Before we continue we need some more notation. We say that a non-negative Radon measure $\mu$ belongs to $\mathcal M_1$ if there exists constant $A$ such that
\[
\int_{\Omega}(-u)\;d\mu\leq A\, e_1(u)^{\frac{1}{n+1}}\, ,
\]
holds for all $u\in \E_1$. Let the functional $\J:\E_1\to \RE$ be defined by
\[
\J(u)=\frac {1}{n+1}\int_{\Omega}(-u)(dd^cu)^n+\int_{\Omega}ud\mu=\frac {1}{n+1} e_1(u)-\|u\|_1\, .
\]

\bigskip

\noindent {\bf Theorem B (Dirichlet's principle).} \emph{Let $\mu\in \mathcal M_1$, and $u\in \E_1$. Then the following
assertions are equivalent}
\begin{enumerate}
\item $\displaystyle\ddcn{u}=d\mu$,
\item $\displaystyle\J(u)=\inf_{w\in\mathcal{E}_1}\J (w)$.
\end{enumerate}

\bigskip

Theorem~B above gives a characterization of solutions $u$ of the Dirichlet problem $(dd^cu)^n=\mu$ as a minimizing functions for the functional $\J$ defined by the measure $\mu$.
This theorem was first proved by Bedford and Taylor in~\cite{bt_var1,bt_var2} for the homogeneous Monge-Amp\`{e}re equation in the class of locally bounded plurisubharmonic functions.  Later Kalina proved the Dirichlet principle
in~\cite{kalina} under some additional assumptions on $\mu$ and $u$. Using that  the first two conditions in
Theorem~A are equivalent, Persson~\cite{persson}, proved this Dirichlet principle in $\E_1$.
Here, we prove Theorem B without using Theorem A.

In the process of writing this note we have not only been inspired
by Bedford's and Taylor's, Kalina's and Persson's pioneer work, but also of the recent work by Berman et al.~\cite{berman_etc}. The authors would also like to express their gratitude to Robert Berman and Sebastien Boucksom for valuable discussions and comments  on an earlier version of this paper.

\section{Preliminaries}

\begin{theorem}\label{thm_holder2} Let $p>0$, and $n\geq 2$. Then there exists a constant $D(n,p)\geq
1$, depending only on $n$ and $p$, such that for any $u_0,u_1,\ldots , u_n\in\Ep$ it holds that
\begin{multline*}
\int_\Omega (-u_0)^p dd^c u_1\wedge\cdots\wedge dd^c u_n \\ \leq D(n,p)\;
e_p(u_0)^{p/(p+n)}e_p(u_1)^{1/(n+p)}\cdots e_p(u_n)^{1/(n+p)}\, .
\end{multline*}
Furthermore, $D(n,1)=1$ and $D(n,p)>1$ for $p\neq 1$.
\end{theorem}
\begin{proof}
See Theorem~3.4 in~\cite{persson} (see also~\cite{czyz_ineq,czyz_energy,cegrell_pc,cegr_pers}).
\end{proof}
It was proved in~\cite{czyz_ineq} (see also~\cite{czyz_modul}) that for
$p\neq 1$ the constant $D(n,p)$ in Theorem~\ref{thm_holder2} is strictly great than $1$. For this
reason we can not use similar variational method to prove the
Dirichlet principle in the class $\Ep$ when $p\neq 1$.

\begin{lemma}\label{coer} For all $u, v \in\E_1$ we have that
\[
\, e_1(u+v)^{\frac{1}{n+1}} \leq \, e_1(u)^{\frac{1}{n+1}} + \, e_1(v)^{\frac{1}{n+1}}\, .
\]
Furthermore, if $\mu\in \mathcal M_1$, then $\J$ is convex, and if $\|u_j\|_1\to \infty$, then $\J(u_j)\to \infty$.
\end{lemma}
\begin{proof} The first statement, triangular type inequality for $ e_1(u)^{\frac{1}{n+1}}$, follows from Theorem~\ref{thm_holder2} since
\[
e_1(u+v) \leq  e_1(u)^{\frac{1}{n+1}}\, e_1(u+v)^{\frac{n}{n+1}} +
 e_1(v)^{\frac{1}{n+1}}\, e_1(u+v)^{\frac{n}{n+1}}\, .
\]
In particular, $\ e_1(u)^{\frac{1}{n+1}}$ is convex and so is $\J$, under the assumption that $\mu\in \mathcal M_1$.
From the definition of $\mathcal M_1$ it follows that there exists constant $A>0$ such that
\[
\|u\|_1\leq Ae_1(u)^{\frac{1}{1+n}}\, ,\qquad \text{for all } u\in \E_1\, .
\]
If $\|u_j\|_1\to \infty$, then $e_1(u_j)\to \infty$, and therefore we get that
\[
\J(u_j)=\frac {1}{n+1}e_1(u_j)-\|u_j\|_1\geq \frac {1}{n+1}e_1(u_j)-Ae_1(u_j)^{\frac {1}{n+1}}\to \infty\, .
\]
This completes this proof.
\end{proof}
\begin{lemma}\label{lem_intineq} Let $v, u \in  \E_1 (\Omega)$, and $w \in \E_1\cap C(\Omega)$.  Then
\begin{equation}\label{lem_intineq1}
\int_{\{w<u \}} (-v)(dd^cu)^n \leq \int\limits_{\{w < u \}}
(-v)(dd^cw)^n\, .
\end{equation}
\end{lemma}
\begin{proof} Assume first that $v, w \in \E_1\cap C(\Omega)$. Without loss of generality we
can assume  that $\int_{\{ u=w\}} (-v)(dd^c w)^n =0$. The measure $(dd^cw)^n$ vanishes on pluripolar sets, and
therefore we have that
\[
 \int_{\{ u=rw\}} (-v)(dd^c w)^n =0\, ,
\]
except for at most denumerably many $r$. Lemma~5.4 in~\cite{cegrell_pc} yields that
\begin{multline*}
\int_{\{w<u \}} (-v)(dd^cu)^n = \int_{\{w < u \}}(-v)(dd^c\max(u,w))^n \\ =
\int_{\Omega} (-v)(dd^c \max(w,u))^n +  \int_{\{w \geq u \}}v(dd^c\max(u,w))^n\\ \leq
\int_{\Omega} (-v)(dd^c w)^n +  \int_{\{w > u \}}v(dd^c\max(u,w))^n \\=
\int\limits_{\{w<u\}} (-v)(dd^c w)^n + \int\limits_{\{ u=w\}} (-v)(dd^c w)^n\, .
\end{multline*}
Thus, inequality~(\ref{lem_intineq1}) holds if $v\in \E_1\cap C(\Omega)$. An approximation of  $v  \in \E_1 (\Omega)$
by a decreasing sequence in $ \E_1\cap C(\Omega)$ completes the proof (see e.g.~\cite{cegrell_gdm,cegrell_approx}).
\end{proof}
\begin{lemma}\label{lip}
Let $u,v\in \E_1$, and assume that $v$ is continuous. For $t < 0$, put
\[
P(u+tv)=\sup\{w\in \E_1:w\leq u+tv\}\, .
\]
Then $P(u+tv)\in \E_1$, and for $s<0$ we have that
\[
|P(u+tv)-P(u+sv)|\leq |t-s|(-v)\, .
\]
\end{lemma}
\begin{proof}
For $t<0$ the function $P(u+tv)$
is upper semicontinuous. Furthermore, $u\leq P(u+tv)\leq u+tv$, and therefore  $P(u+tv)\in \E_1$. For $s<t<0$, we
have that
\[
P(u+tv)\leq P(u+sv)\, , \;\;\text{ and }\;\;
P(u+sv)+(t-s)v\leq P(u+tv)\, .
\]
Thus, $|P(u+tv)-P(u+sv)|\leq |t-s|(-v)$.
\end{proof}
\begin{lemma}\label{comarison principle} Let $u,v\in \E_1$, and assume that $v$ is continuous.
For any $0\leq k \leq n,$
\begin{equation}\label{cp_1}
\lim_{t\nearrow 0}\int_{\Omega} \frac{P(u+tv) - tv -u}{t}(dd^cu)^k\wedge (dd^cP(u+tv))^{n-k} = 0\, .
\end{equation}
In particular
\begin{equation}\label{cp2}
\lim_{t\nearrow 0}\int_{\Omega} \frac{P(u+tv)  -u}{t}(dd^cu)^k\wedge (dd^cP(u+tv))^{n-k} = \int_{\Omega} v(dd^cu)^n\, .
\end{equation}
\end{lemma}
\begin{proof}
Consider the  function $h(t)=\frac{P(u+tv)-tv-u}{t}$, for $t<0$. A straightforward calculation shows that $h$ is a decreasing function, and
\[
0\leq \frac{P(u+tv)-tv-u}{t} \leq -v\, .
\]
Hence, for fixed $s<0$ we have that
\begin{multline*}
\lim_{t\nearrow 0}\int_{\Omega} \frac{P(u+tv) - tv -u}{t}(dd^cu)^k\wedge (dd^cP(u+tv))^{n-k}\\ \leq
\lim_{t\nearrow 0}\int_{\Omega} \frac{P(u+sv) - sv -u}{s}(dd^cu)^k\wedge (dd^cP(u+tv))^{n-k}\\ =
\int_{\Omega} \frac{P(u+sv)-sv-u}{s}(dd^cu)^n \leq \int\limits_{\{P(u+sv) - sv<u \}}(-v) (dd^cu)^n\,
.
\end{multline*}
Let $u_k\in \E_0\cap C(\Omega)$ be a decreasing sequence that converges to $u$ such that
\[
\int\limits_{\{P(u+sv)-sv<u \}} (-v)(dd^cu)^n \leq 2\int\limits_{\{P(u_k+sv)-sv < u \}}
(-v)(dd^cu)^n\, .
\]
We can apply  Lemma~\ref{lem_intineq} and conclude that
\begin{multline*}
\int\limits_{\{P(u_k+sv) - sv < u \}}(-v) (dd^cu)^n\\ \leq
\int\limits_{\{P(u_k+sv) - sv < u_k \}}(-v) (dd^c(P(u_k+sv) - sv))^n \leq -s M \to 0\, , \qquad \text{as } s \to 0\, .
\end{multline*}
Here $M$ is a constant only depending on $n$, $\|v\|$, and $\int_{\Omega} v(dd^c(u+v))^n$. We have used that
\[
\int\limits_{\{P(u_k+sv)  < u_k + sv\}} (dd^c(P(u_k+sv)))^n =0\, .
\]
This is a consequence of Corollary 9.2 in~\cite{bt}. The equality (\ref{cp2}) is an  consequence of the equality (\ref{cp_1}).
The proof is complete.
\end{proof}
\begin{lemma}\label{dif_pos}
Let $u,v\in \E_1$, and assume that $v$ is continuous. For $t>0,$ we set
\[
f(t)=\int_{\Omega}-(u+tv)(dd^c(u+tv))^n=e_1(u+tv)\, .
\]
Then
\[
f'(0^+)=(n+1)\int_{\Omega}(-v)(dd^cu)^n\, .
\]
\end{lemma}
\begin{proof} This is an immediate consequence of the construction.
\end{proof}
\begin{lemma}\label{dif_neg}
Let $u,v\in \E_1$, and assume that $v$ is continuous. For $t<0$, we set
\[
f(t)=\int_{\Omega}(-P(u+tv))(dd^cP(u+tv))^n=e_1(P(u+tv))\, .
\]
Then
\[
f'(0^-)=(n+1)\int_{\Omega}(-v)(dd^cu)^n\, .
\]
\end{lemma}
\begin{proof}
Note that
\[
\begin{aligned}
&\frac 1t\left (\int_{\Omega}(-P(u+tv))(dd^cP(u+tv))^n-\int_{\Omega}(-u)(dd^cu)^n\right )=\\
&=\sum_{k=0}^n\int_{\Omega} \frac{u-P(u+tv)}{t}(dd^cu)^k\wedge (dd^cP(u+tv))^{n-k},
\end{aligned}
\]
and Lemma~\ref{comarison principle} completes the proof.
\end{proof}
\begin{corollary}\label{cor} Let $u,v\in \E_1$, and assume that $v$ is continuous. Then it holds that
\begin{multline*}
\J (P(u+tv))'(0^-)=\Bigg(\frac {1}{n+1}\int_{\Omega}(-P(u+tv))(dd^cP(u+tv))^n\\ +\int_{\Omega}P(u+tv)d\mu\Bigg)'(0^-) \geq
\int_{\Omega}(-v)(dd^cu)^n+\int_{\Omega}vd\mu\, .
\end{multline*}
\end{corollary}
\begin{proof}
The existence of $\J (P(u+tv))'(0^-)$ follows from Lemma~\ref{dif_neg} and the fact that the
function $t\to \frac {P(u+tv)-u}{t}$ is decreasing. For $t<0$
\[
\int_{\Omega}\frac{(P(u+tv) -u)}{t}d\mu=\int_{\Omega}\frac{(P(u+tv)-tv -u)}{t}d\mu +
\int_{\Omega}vd\mu \geq \int_{\Omega}vd\mu\, ,
\]
and the proof is finished by Lemma~\ref{dif_neg}.
\end{proof}

\section{Proof of the Theorem B}\label{priciple}

\begin{proof}
Let $\mu\in \mathcal M_1$, and $u\in \E_1$.

$(1)\Rightarrow(2)$: Assume that $(dd^cu)^n=d\mu$, and let $v\in \E_1$. Then by Theorem~\ref{thm_holder2}, and Young's inequality we get that
\[
\int_{\Omega}(-v)d\mu=\int_{\Omega}(-v)(dd^cu)^n\leq e_1(v)^{\frac {1}{n+1}}
e_1(u)^{\frac {n}{n+1}} \leq
\frac {1}{n+1}e_1(v) + \frac {n}{n+1}e_1(u)\, .
\]
Hence
\[
\J(v)=\frac {1}{n+1}e_1(v)+\int_{\Omega}vd\mu\geq -\frac {n}{n+1}e_1(u)=\J(u)\, .
\]
Thus,  $\J(u)=\inf_{w\in\mathcal{E}_1}\J (w)$.

$(2)\Rightarrow(1)$: Let $u\in \E_1$ be such that $\J(u)=\inf_{w\in\mathcal{E}_1}\J (w)$. Take an arbitrary function $v$ in  $\E_1\cap C(\Omega)$, and define
\[
g(t)=\frac {1}{n+1}\int_{\Omega}(-P(u+tv))(dd^cP(u+tv))^n+\int_{\Omega}P(u+tv)d\mu\, , t<0
\]
and $g(t)=\J(u+tv)$ for  $t\geq 0$. Since, by assumption, for all t, $g(0)\leq g(t)$, it follows
that $0 \geq g^\prime (0^-)$,  and $g^\prime(0^+)\geq 0$. The existence of $g'(0^+)$ and $g'(0^-)$
follows from Lemma~\ref{dif_pos} and Corollary~\ref{cor} respectively. The last inequality and
Lemma~\ref{dif_pos} gives us that
\[
\int_{\Omega}(-v)(dd^cu)^n+\int_{\Omega}vd\mu \geq 0\, ,
\]
and therefore it follows from Corollary~\ref{cor} that $g^\prime (0^-)=0$, and
\[
0=\int_{\Omega}(-v)(dd^cu)^n+\int_{\Omega}vd\mu\, .
\]
Thus,
\[
\int_{\Omega}vd\mu=\int_{\Omega}v(dd^cu)^n\, .
\]
Since $v$ was arbitrary, we conclude, by Lemma~3.1 in~\cite{cegrell_gdm}, that $(dd^cu)^n=d\mu$.
\end{proof}
\begin{remark}
The uniqueness of the solution for the equation $\ddcn{u}=d\mu$ follows from the comparison principle (see
 e.g.~\cite{czyz_cegrell_hiep,cegrell_pc}).  Using Lemma~\ref{coer}, uniqueness in $\E_1$ can  be obtained
in the following way.
\end{remark}
\begin{proposition}\label{prop}  For any $\mu\in \M_1$ there exists at most  one  function $u\in\E_1$ for which the
functional $\J$ achieves its infimum on $\E_1$. In other words, there exists at most one
solution $u\in\E_1$ for the complex Monge-Amp\`{e}re equation $\ddcn{u}=\mu$.
\end{proposition}
\begin{proof} Let $S$ denotes the set of solutions of the Dirichlet problem for the measure
$\mu$. Then we know by Lemma~\ref{coer} that S is a convex set. Assume that there exist
functions $u,v\in S.$
Then also $tu+(1-t)v\in S$, for $0\leq t \leq 1$. We also
have that for any $1\leq k \leq n$, and all $\varphi\in\Eo$, it holds that
\begin{multline*}
\int_\Omega (-\varphi) (dd^c u)^k\wedge (dd^c v)^{n-k}\leq \left(\int_\Omega (-\varphi)\ddcn{u}\right)^{\frac{k}{n}}\left(\int_\Omega (-\varphi)\ddcn{v}\right)^{\frac{n-k}{n}}\\
 \leq \int_\Omega (-\varphi)d\mu\, .
\end{multline*}
This implies that for all $1 \leq k \leq n$ we have that
\begin{equation}\label{prop_1}
(-\varphi) (dd^cu)^k\wedge (dd^c v)^{n-k}=(-\varphi) d\mu\, ,
\end{equation}
since otherwise we would have that $tu+(1-t)v\not\in S$. From~(\ref{prop_1}) it follows that
\[
(dd^cu)^k\wedge (dd^c v)^{n-k}= d\mu\, .
\]
Now we can use an argument from the proof of Theorem~3.15 in~\cite{cegrell_bdd} to prove that $u =
v$. By~\cite{cegrell_approx}, there exists a strictly plurisubharmonic exhaustion function
$\psi\in\mathcal E_{0}\cap C^{\infty}(\Omega)$ for $\Omega$. It is enough to show that
\[
\int_{\Omega} d(u-v)\wedge d^c(u-v)\wedge(dd^c\psi)^{n-1} = 0.
\]
It is easy to see that
\[
0=\int_{\Omega} d(u-v)\wedge d^c(u-v)\wedge(dd^cu)^{a}\wedge(dd^cv)^b\wedge dd^c\psi,
\]
for $a+b=n-2$.

Assume that
\[
0=\int_{\Omega} d(u-v)\wedge d^c(u-v)\wedge(dd^cu)^{a}\wedge(dd^cv)^b\wedge (dd^c\psi)^p,
\]
for $a+b=n-1-p$. Then, for $ a+b=n-2-p$ we have
\begin{multline*}
0\leq \int_{\Omega} d(u-v)\wedge d^c(u-v)\wedge(dd^cu)^{a}\wedge(dd^cv)^b\wedge (dd^c\psi)^{p+1}\\ =
\int_{\Omega} -(u-v)dd^c(u-v)\wedge(dd^cu)^{a}\wedge(dd^cv)^b\wedge (dd^c\psi)^{p+1}\\ =
\int_{\Omega} -\psi (dd^c(u-v))^2\wedge(dd^cu)^{a}\wedge(dd^cv)^b\wedge (dd^c\psi)^{p}\\=
\int_{\Omega} d\psi \wedge d^c(u-v)\wedge dd^c(u-v)\wedge(dd^cu)^{a}\wedge(dd^cv)^b\wedge
(dd^c\psi)^{p}\\ \leq
\left|\int_{\Omega} d\psi \wedge d^c(u-v)\wedge dd^cu\wedge(dd^cu)^{a}\wedge(dd^cv)^b\wedge (dd^c\psi)^{p}\right|\\ +
\left|\int_{\Omega} d\psi \wedge d^c(u-v)\wedge dd^cv\wedge(dd^cu)^{a}\wedge(dd^cv)^b\wedge
(dd^c\psi)^{p}\right|\\ \leq
\Bigg(\int_{\Omega} d\psi \wedge d^c\psi\wedge(dd^cu)^{a+1}\wedge(dd^cv)^b\wedge (dd^c\psi)^{p}\times \\ \times
\int_{\Omega} d(u-v) \wedge d^c(u-v)\wedge(dd^cu)^{a+1}\wedge (dd^cv)^b\wedge
(dd^c\psi)^{p}\Bigg)^{\frac{1}{2}}\\+
\Bigg(\int_{\Omega} d\psi \wedge d^c\psi\wedge(dd^cu)^{a}\wedge(dd^cv)^{b+1}\wedge (dd^c\psi)^{p}\times \\ \times
\int_{\Omega} d(u-v) \wedge d^c(u-v)\wedge(dd^cu)^{a}\wedge(dd^cv)^{b+1}\wedge
(dd^c\psi)^p\Bigg)^{\frac{1}{2}} =0\, .
\end{multline*}
\end{proof}

\section{Proof of Theorem A}
 We begin with a lemma.

\begin{lemma}\label{convergence} Let $\mu$ be a non-negative Radon measure such that $\mu (\Omega)<+\infty$. If there exists a constant $A>0$ such that
\begin{equation}\label{eq}
\int_{\Omega}(-\varphi)^2\, d\mu\leq A\, e_1(\varphi)^{\frac2{n+1}}\text{ for all } \varphi\in\E_1\, ,
\end{equation}
then $\mu\in \M_1$. Furthermore, for any sequence $\{v_j\}\subset\E_1$ such that
$e_1(v_j)\leq 1$, there exists a subsequence $\{v_{j_k}\}$  convergent in the $L^1(\mu)$ topology.
Finally, there exists a uniquely determined function $u\in\mathcal E_1,$ with  $(dd^cu)^n = \mu.$

\end{lemma}
\begin{proof} Assume that $\mu$ is a non-negative Radon measure with $\mu (\Omega)<+\infty$, and take a
function $\varphi\in\E_1$. Then it follows from inequality~(\ref{eq}) that there exists a constant $A>0$ such that
\begin{multline}
\int_{\Omega} (-\varphi)\,d\mu\leq \left(\int_{\Omega}
(-\varphi)^2\,d\mu\right)^{1/2}\mu(\Omega)^{1/2} \leq A^{1/2}
e_1(\varphi)^{\frac1{n+1}}\mu(\Omega)^{1/2}\\ = C e_1(\varphi)^{\frac1{n+1}}<+\infty\, , \qquad \text{where } C=A^{1/2}\mu(\Omega)^{1/2}\,
\end{multline}
Thus, $\E_1\subseteq L^1(\mu)$.

Assume now that $\{v_j\}\subset\E_1$ is a sequence such that
\[
\sup_j e_1(v_j)\leq 1\, .
\]
We can then pick a subsequence, again denoted by $v_j$, convergent as distributions to $v\in\mathcal E_1$
and such that sequence  $\{v_j\,d\mu\}$ is weakly convergent to some measure $\nu$. Then we have that
by inequality~(\ref{eq}) that there exists a constant $A>0$ such that
\[
\int_{\Omega} (-v_j)^2\,d\mu\leq A e_1(v_j)^{\frac{2}{n+1}}\leq A\, .
\]
Thus, $v_j\in L^2(\mu)$. Therefore, there exists a finite convex combination of $v_j$, denote this by $w_j$,
such that $\{w_j\}\subset\E_1$ converges to some function $w\in L^2(\mu)$. Furthermore, $d\nu=w d\mu$. But $\{v_j\}$ is
weakly convergent w.r.t. Lebesgue measure to $(\limsup v_j)^*$, and therefore is $\{w_j\}$ is weak convergent to
$(\limsup w_j)^* =(\limsup v_j)^*$. Hence, $w=(\limsup v_j)^*$.

To complete the proof of the lemma,
we show that then there exists a minimizer. So by Theorem~B there exists a function $u\in\E_1$ such that $\ddcn{u}=\mu$. Let  $\{u_j\}$ be a sequence in $\E_1$ such that  $\lim_{j\to \infty}\J(u_j)=\inf_{w\in\mathcal{E}_1}\J (w)$.
Using  Lemma~\ref{coer} together with what we just have proved,
we  can pick a subsequence again denoted by $\{u_j\}$ and $ u\in\E_1$ such  that
\[
 \int_{\Omega}|u_j-u|d\mu\to 0\qquad \text{as }  j\to\infty\, .
 \]
Set $v_k=(\sup_{j\geq k}u_j)^*$, where $(w)^*$ denotes the upper semicontinuous regularization of $w$. Then it follows that $v_k\geq u_k$, which implies that $e_1(v_k)\leq e_1(u_k)$ (see e.g.~\cite{cegrell_pc} or Lemma~6.1 in~\cite{czyz_modul}). Thus, $v_k\in \E_1$. The decreasing sequence  $\{v_k\}$ converges to $u$, as $j\to \infty$, and $e_1(v_k)\to e_1(u)$, as $k\to\infty$.  The monotone convergence theorem implies that $\int_{\Omega}v_kd\mu \to \int_{\Omega}ud\mu$. Therefore, we have that
\[
e_1(u)=\lim_{j\to\infty}e_1(v_j)=\liminf_{j\to\infty}e_1(v_j)\leq \liminf_{j\to\infty}e_1(u_j)\, .
\]
Hence,
\[
\liminf_{j\to\infty}\J(u_j)=\liminf_{j\to\infty}e_1(u_j)+
\lim_{j\to\infty}\int_{\Omega}u_jd\mu\geq e_1(u)-\|u\|_1=\J(u)\, .
\]
Thus $u$ is a minimizer which completes the proof of the lemma.
\end{proof}
\begin{proof}[Proof of Theorem A]

{\bf (1) $\Rightarrow$ (2)} Assume that there exists a function $u\in\E_1$ such that $\ddcn{u}=\mu$, and take
$\varphi\in\E_1$. Then by Theorem~\ref{thm_holder2} it follows that
\[
\int_{\Omega} (-\varphi)\, d\mu=\int_{\Omega} (-\varphi)\ddcn{u}\leq e_1(\varphi)^{\frac1{n+1}}
e_1(u)^{\frac{n}{n+1}}\, .
\]
Hence, (2) holds with $B=e_1(u)^{\frac{n}{n+1}}$.

\medskip

{\bf (3) $\Rightarrow$ (2)} Assume that condition $(2)$ is not satisfied, then for all $j$ there
exists $u_j\in \E_1$ such that
\[
\int_{\Omega}(-u_j)\;d\mu\geq je_1(u_j)^{\frac{1}{1+n}}.
\]
Without loss of generality we can assume that $e_1(u_j)=1$. Let $v_j=\sum_{l=1}^j\frac {1}{l^2}u_l$,
then by Lemma~\ref{coer}
\[
e_1(v_j)^{\frac {1}{n+1}}\leq \sum _{l=1}^j\frac {1}{l^2}e_1(u_l)^{\frac {1}{n+1}}=\sum
_{l=1}^j\frac {1}{l^2},
\]
which implies that $v=\sum_{l=1}^{\infty}\frac {1}{l^2}u_l\in \E_1$. On the other hand
\[
\int_{\Omega}(-v)\,d\mu\geq \sum_{l=1}^{\infty}\frac {1}{l^2}\int_{\Omega}(-u_l)\,d\mu\geq
\sum_{l=1}^{\infty}\frac {l}{l^2}=\infty,
\]
so $v\notin L^1(\mu)$. This ends the proof.

\medskip

{\bf (2) $\Rightarrow$ (1)}  We follow here the proof of Theorem 5.1 in ~\cite{cegrell_pc}. Assume that $\mu$ is a non-negative Radon measure such that~(\ref{eq2}) holds. Assume first that $\mu$ is supported by a compact set $K\Subset\Omega$, and let $h_K$ denote the relative extremal function for $K$. Set
\[
\M=\left\{\nu\geq 0: \supp \nu\subset K,\;\int_{\Omega}(-\varphi)^2 d\nu\leq Ce_1(\varphi)^{\frac{2}{n+1}}\text{ for all } \varphi\in\E_1\right\}
\]
where $C> 2 e_1(h_K)^{\frac{n-1}{n+1}}$ is a fixed constant. For a compact set $L\subset K$ we have that
\[
h_K\leq h_L\, , \qquad \text{ and } e_1(h_L)\leq e_1(h_K)\, .
\]
Therefore, it follows that
\begin{multline*}
\int_{\Omega} (-\varphi)^2\ddcn{h_L}\leq 2\|h_L\|\int_{\Omega} (-\varphi)(dd^c\varphi)\wedge (dd^c h_L)^{n-1}\\ \leq
2\left(\int_{\Omega} (-\varphi)\ddcn{\varphi}\right)^{\frac{2}{n+1}}
\left(\int_{\Omega} (-h_L)\ddcn{h_L}\right)^{\frac{n-1}{n+1}}\\ \leq C e_1(\varphi)^{\frac{2}{n+1}}\text{ for all } \varphi\in\E_1\, .
\end{multline*}
Hence, for every compact set $L\subset K$ we have that $\ddcn{h_L}\in\M$.

Fix $\nu_0\in \M$ and define
\begin{multline*}
\M^\prime=\Bigg\{ \nu\geq 0: \nu(\Omega)=1,\; \supp \nu\subset K, \\
\int_{\Omega}(-\varphi)^2 d\nu\leq \left(\frac{C}{T} + \frac{C}{\nu_0(\Omega)}\right)
e_1(\varphi)^{\frac{2}{n+1}} \text{ for all } \varphi\in\E_1
\Bigg\}\, ,
\end{multline*}
where $T=\sup\{\nu(\Omega):\nu\in\M\}$. Then we have for $\nu\in \mathcal M$ that
\begin{multline*}
\int_{\Omega} (-\varphi)^2\frac{(T-\nu(\Omega))d\nu_0 +\nu_0(\Omega)d\nu)}{T\nu_0(\Omega)} \leq
\int_{\Omega} (-\varphi)^2d\nu_0\frac{T-\nu(\Omega)}{T\nu_0(\Omega)}+\frac1{T}\int_{\Omega} (-\varphi)^2d\nu\\ \leq
\left(C\frac{T-\nu(\Omega)}{T\nu_0(\Omega)} + \frac{C}{T}\right)e_1(\varphi)^{\frac{2}{n+1}}\leq
\left(\frac{C}{\nu_0(\Omega)} + \frac{C}{T}\right)e_1(\varphi)^{\frac{2}{n+1}}\qquad  \text{ for all } \varphi\in\E_1\, .
\end{multline*}
Hence,
\[
\frac{(T-\nu(\Omega))\nu_0 +\nu_0(\Omega)\nu)}{T\nu_0(\Omega)}\in\M^\prime \text{ for all } \nu\in\M\, .
\]
Thus, $\M^\prime$ is a convex and weak$^*$-compact set of probability measures. Then it follows
from~\cite{rain} that there exist a function $f\in L^1(\mu)$ and a measure $\nu\in\M^\prime$ such
that $\mu=f\,d\nu+\nu_s$, where $\nu_s$ is orthogonal to $\M^\prime$. Note that since
$(dd^ch_L)^n\in \mathcal M$, then all measures orthogonal to $\M^\prime$ must be supported on
pluripolar sets. Therefore $\nu_s\equiv 0$, since $\mu$ vanishes on pluripolar sets. Note also that by Lemma~\ref{convergence}, we know that for $\mu\in \M^\prime$ there exists a uniquely determined function $u\in \E_1$ such that $(dd^cu)^n=\mu$.

Set $\mu_j=\min(f,j)d\nu$. Since $\nu$ satisfies inequality~(\ref{eq}), then so do $\mu_j$. Therefore there
a unique $u_j\in\E_1$ with $\ddcn{u_j}=d\mu_j$. Since $\mu_j(\Omega)<+\infty$, we can use Theorem~4.5 in~\cite{cegrell_pc} to see that $\{u_j\}$ is a decreasing sequence that converges to a function $u\in\E_1$ with $\ddcn{u}=\mu$.

 Finally, if $\mu$ only satisfy~(\ref{eq2}), let $\{K_j\}$ be an increasing  sequence compact subsets of $\Omega$
 with union equal to $\Omega$ and set $\mu_j=\chi_{K_j}d\mu$. We can complete the proof as above.

\medskip

{\bf (1) $\Rightarrow$ (4)} Assume that there exists a function $u\in\E_1$ such that $\ddcn{u}=\mu$.
Let $\{v'_j\}\subset\E_1$ be a sequence with $e_1(v'_j)\leq 1$, and  $\{v'_j\}$ converges
as distributions and select a weak*-convergent subsequence of $\{v'_jd\mu\}$ denoted by $\{v_jd\mu\}$, converging to $d\nu$ for some measure $\nu\leq 0$. Let $\psi\in C^{\infty}_0(\Omega)$, $0\leq \psi\leq
1$, then $\psi\ddcn{u}\leq \ddcn{u}$.  Hence, the measure $\psi\ddcn{u}$ satisfies condition (2),
and therefore also (1). Thus, there exists a function $\varphi\in\E_1$ such that $\ddcn{\varphi}=\psi\ddcn{u}$
and from the proof of (2) implies (1) it follows  that $u\leq \varphi$. This, together with Theorem~\ref{thm_holder2} yields that
\begin{multline*}
\int_{\Omega}(-v_j)\psi d\mu=\int_{\Omega}(-v_j)\ddcn{\varphi}\leq e_1(v_j)^{\frac{1}{n+1}}\left(\int_{\Omega}(-\varphi)\ddcn{\varphi}\right)^{\frac{n}{n+1}}\\ \leq
e_1(v_j)^{\frac{1}{n+1}}\left(\int_{\Omega}(-u)\psi\ddcn{u}\right)^{\frac{n}{n+1}}
=e_1(v_j)^{\frac{1}{n+1}}\left(\int_{\Omega}(-u)\psi d\mu\right)^{\frac{n}{n+1}}\, .
\end{multline*}
This means that $d\nu$ is absolutely continuous w.r.t. $d\mu$, and therefore there exists a function $f\in L^1(\mu)$, $f\leq 0$, such that $d\nu=f d\mu$. Let now $\psi\in L^{\infty}(\Omega), 0\leq \psi\leq 1$. By a similar argument as above,
\begin{equation}\label{proofA_1}
\int_{\Omega}\psi f\, d\mu= \lim_{j\to+\infty} \int_{\Omega} \psi v_{j}\, d\mu
\end{equation}
(Choose   $\psi_{\varepsilon} \in C^{\infty}_0(\Omega)$ so that  $\int |\psi - \psi_{\varepsilon}|(-f-u)d\mu \leq \varepsilon$
and continue  as above.)

Hence, there exist finite convex combinations of $v_j$, denoted by $w_j$, that converges to $f$
in $L^1 (\mu)$. Therefore there exists a subsequence $\{w_{j_k}\}$ of $\{w_j\}$ that converges to $f$ a.e. $[\mu]$. From now on we shall use the notation $\{w_j\}$ instead of $\{w_{j_k}\}$. Set $v=\lim_{j\to+\infty} \left(\sup_{k\geq j} v_k\right)^*$, then it follows from Fatou's lemma that $f\leq v$. Furthermore, we get that $\lim_{j\to+\infty} \left(\sup_{k\geq j} w_k\right)^*=v$
since $\{v'_j\}, \{v_j\}$ and $\{w_j\}$  converges to the same limit as distributions, $v\in\E_1$, and
\[
\int_{\Omega}v d\mu=\int_{\Omega}\lim_{j\to+\infty} \left(\sup_{k\geq j} w_k\right)^* d\mu
=\int_{\Omega}\lim_{j\to+\infty} \left(\sup_{k\geq j} w_k\right) d\mu=\int_{\Omega}f d\mu\, .
\]
Thus, $v=f= \lim_{j\to+\infty} \left(\sup_{k\geq j} v'_k\right)$ a.e. $[\mu]$ and it follows from~(\ref{proofA_1}) that
\[
\lim_{j\to\infty}\int_{\Omega}v_j d\mu=\int_{\Omega}v\, d\mu\, .
\]
Then
\[
\int_{\Omega}|v_j-v| d\mu\leq \int_{\Omega}\left(\max_{k\geq j} v_k -v_j\right)d\mu +
\int_{\Omega}\left(\max_{k\geq j} v_k -v\right)d\mu =I_1+ I_2\, .
\]
But $I_1$ converges to $0$, as $j\to \infty$, since
\[
\lim_{j\to\infty}\int_{\Omega}\max_{k\geq j} v_k\,  d\mu=\lim_{j\to\infty}\int_{\Omega}v_jd\mu\, ,
\]
and $I_2$ converges to $0$ by the monotone convergence theorem. Thus,
\[
\lim_{j\to \infty}\int_{\Omega}|v_j-v| d\mu=0\, ,
\]
i.e. $v_j$ converges to $v$ in $L^1(\mu)$. Thus, every subsequence of $\{v'_j\}$ contains a subsequence that converges to
$v=f= \lim_{j\to+\infty} \left(\sup_{k\geq j} v'_k\right)$  in $L^1(\mu)$. This ends the proof of Theorem A.
\end{proof}


\begin{thebibliography}{30}

\bibitem{czyz_cegrell_hiep} \AA hag P., Cegrell U., Czy\.z R. and   Ph\d{a}m H. H.,
Monge-Amp\`{e}re measures on pluripolar sets, J. Math. Pures Appl. 92 (2009), 613-627.

\bibitem{cegrell_hiep_etc} \AA hag P., Cegrell U., Ko{\l}odziej S., Ph\d{a}m H.H. and Zeriahi A., Partial pluricomplex
energy and integrability exponents of plurisubharmonic functions, Adv. Math. 222 (2009), 2036-2058.

\bibitem{czyz_ineq} \AA hag P. and Czy\.z R., An inequality for the beta function with application to pluripotential theory, J. Inequal. Appl. 2009, Art. ID 901397, 8 pp.

\bibitem{czyz_modul} \AA hag P. and Czy\.z R., Modulability and duality of certain cones in pluripotential theory,
J. Math. Anal. Appl. 361 (2010), 302-321.

\bibitem{czyz_energy} \AA hag P. and Czy\.z R. and Ph\d{a}m H.H., Concerning the energy
class $\Ep$ for $0<p<1$,  Ann. Polon. Math. 91 (2007), 119-130.

\bibitem{bt_var1} Bedford E. and Taylor B.A., Variational properties of the complex Monge-Amp\`{e}re equation.
I. Dirichlet principle, Duke Math. J. 45 (1978), 375-403.

\bibitem{bt_var2} Bedford E. and Taylor B.A., Variational properties of the complex Monge-Amp\`{e}re equation.
II. Intrinsic norms, Amer. J. Math. 101 (1979), no. 5, 1131-1166.

\bibitem{bt} Bedford E. and Taylor B.A., A new capacity for plurisubharmonic functions, Acta Math 149
(1982), 1-40.

\bibitem{berman_etc} Berman R. J., Boucksom S., Guedj V. and Zeriahi A., A variational approach to complex Monge-Amp\`ere
equations, arXiv:0907.4490 (2009).

\bibitem{cegrell_pc} Cegrell U., Pluricomplex energy, Acta Math.
180 (1998), 187-217.

\bibitem{cegrell_gdm} Cegrell U., The general definition of the complex
Monge-Amp\`{e}re operator, Ann. Inst. Fourier (Grenoble)  54
(2004), 159-179.


\bibitem{cegrell_approx} Cegrell U., Approximation of plurisubharmonic function in hyperconvex domains, Acta Universitatis Upsaliensis 86, Proceedings of the conference in honour of C. Kiselman ("Kiselmanfest", Uppsala, May 2006), 125-129.

\bibitem{cegrell_bdd} Cegrell U.,  A general Dirichlet problem for the complex Monge-Amp\`{e}re operator, Ann. Polon. Math.  94 (2008), 131-147.


\bibitem{cegrell_sub} Cegrell U., Ko{\l}odziej S. and Zeriahi A.,
Subextension of plurisubharmonic functions with weak singularites,
Math. Z. 250 (2005), 7-22.

\bibitem{cegr_pers} Cegrell U. and Persson L., An energy estimate for the
complex Monge-Amp\`{e}re operator, Ann. Polon. Math. 67 (1997), 95-102.

\bibitem{czyz} Czy\.{z} R., The complex Monge-Amp\`{e}re operator in the Cegrell
classes, Dissertationes Math. 466 (2009), 83 pp.

\bibitem{diller}  Diller J., Dujardin R. and Guedj V., Dynamics of meromorphic maps with small topological degree II:
 Energy and invariant measure, Commentarii Math. Helvetici (in press).

\bibitem{kalina} Kalina J., Some remarks on variational properties of inhomogeneous
complex Monge-Amp\`{e}re equation, Bull. Polish Acad. Sci. Math. 31 (1983), 9-13.

\bibitem{klimek} Klimek M., Pluripotential theory, The Clarendon Press,
Oxford University Press, New York, 1991.

\bibitem{kolo_mem} Ko{\l}odziej S., The complex Monge-Amp\`{e}re equation and
pluripotential theory, Mem. Amer. Math. Soc. 178 (2005).

\bibitem{persson} Persson L., A Dirichlet principle for the
 complex Monge-Amp\`{e}re operator, Ark. Mat. 37 (1999), 345-356.


\bibitem{rain} Rainwater J., A note on the preceding paper, Duke Math. J. 36 (1969), 799-800.

\end{thebibliography}
\end{document}